\documentclass[a4paper,10pt]{article}
\usepackage[utf8]{inputenc}

\usepackage{amsfonts}
\usepackage{color,enumerate,wrapfig,amssymb}
\usepackage{amsmath}
\usepackage{amsthm}
\usepackage{esint}
\usepackage{mathrsfs}

\newtheorem{theorem}{Theorem}
\newtheorem{lemma}{Lemma}
\newtheorem{proposition}{Proposition}
\newtheorem{corollary}{Corollary}

\newtheorem{example}{Example}
\newtheorem{definition}{Definition}

\begin{document}

\title{Optimal regularity in the optimal switching problem}

\author{Gohar Aleksanyan}

\maketitle

\begin{abstract}
 
 In this article we study the optimal regularity for solutions to the following weakly coupled system
 with interconnected obstacles
 \begin{equation*}
\begin{cases}
\min (-\Delta u^1+f^1, u^1-u^2+\psi^1)=0 \\
\min (-\Delta u^2+f^2, u^2-u^1+\psi^2)=0,
\end{cases}
\end{equation*}
 arising in the optimal switching problem with two modes.

\par
We derive the optimal $C^{1,1}$-regularity for the minimal solution under the assumption that
the zero loop set $\mathscr{L}:= \{ \psi^1+\psi^2=0\}$ is the closure of its interior. 
This result is optimal and we provide a counterexample showing that the $C^{1,1}$-regularity does not hold 
without the assumption $\mathscr{L} = \overline{ \mathscr{L}^0} $.

\end{abstract}

\section{Introduction}
We consider the following system of weakly coupled equations of obstacle type
\begin{equation}  \label{system2}
\begin{cases}
\min (-\Delta u^1+f^1, u^1-u^2+\psi^1)=0 \\
\min (-\Delta u^2+f^2, u^2-u^1+\psi^2)=0,
\end{cases}
\end{equation}
with given Dirichlet boundary conditions $ u^i=g^i \textrm{ on } \partial\Omega$.
These type of systems arise in optimal switching problems with two switching modes.
Here $f^1$ and $f^2 $ are the running cost functions corresponding to the switching modes.
The functions $\psi^1$ and $\psi^2$ are the costs of switching 
from one mode to the other. More details on the optimal switching problem are provided in Section 2.1.

\par
The uniqueness and $C^{1,1}$-regularity of the solutions to such systems 
have been studied in the literature under the assumption 
that the switching costs are nonnegative constants, \cite{EF79}, \cite{LB83}, \cite{BL86}. 
In the paper \cite{CGT13}
the regularity of the solutions to an obstacle type weakly coupled system with first 
order Hamiltonians
is studied using adjoint methods under the assumption that 
each of the switching costs 
is bounded from below by a positive constant.

\par
In our paper we make only the nonnegative loop assumption. This is a necessary 
condition for the system to be well-defined. Indeed, let $(u^1, u^2)$ be a solution to \eqref{system2}, then
$u^1-u^2+\psi^1 \geq 0$ and $u^2-u^1+\psi^2 \geq 0 $, which implies
\begin{equation} \label{loop}
 \psi^1 (x) + \psi^2 (x) \geq 0.
\end{equation}

\par
In the optimal switching setting, the condition \eqref{loop}
prevents the agent from making arbitrary gains by looping, in the sense that
$\psi^1 (x) + \psi^2 (x)$
is the cost of switching from one mode to the other and immediately
switching back. We denote the set where it is possible to switch for 
free by 
\begin{equation*}
  \mathscr{L}=\{x \in \Omega \mid \psi^1 (x) + \psi^2 (x) = 0\},
\end{equation*}
and call it free switching or zero loop set.

\par
By using the penalization/regularization method we derive the existence of solutions,
showing that through a subsequence the solutions of the penalized system converge to the minimal solution
$ ( u^1_0,u^2_0)$ to \eqref{system2}.
Then we see that the solution $u^i_0 \in C^{1,\gamma}$,
for every $0<\gamma <1$ and
\begin{equation} \label{bddlap}
 \lVert \Delta u^i_0 \rVert_{L^\infty(\Omega)} \leq 
 \max_i \Arrowvert \Delta \psi^i \rVert_{L^{\infty}(\Omega)}+
 3 \max_i \Arrowvert f^i \rVert_{L^{\infty}(\Omega)}.
\end{equation}

\par
The aim of the paper is to investigate if the solutions are $C^{1,1}$, which is the best regularity
that we can hope that the solutions achieve. The structure of our system shows that at some subdomains of $\Omega$, 
the regularity of the solutions can be derived by already known $C^{1,1}$-regularity results for the
obstacle problem. In our discussion we see that the main point is to describe the regularity 
at so called meeting points lying on $\partial \mathscr{L}$, the boundary of the zero loop set.

\par
In the main theorem, Theorem 4, we show that at the meeting points  $x_0 \in \partial \mathscr{L}^0 \cap \Omega$
the solutions are $C^{2, \alpha}$, 
under the assumption that $f^i \in C^\alpha $ and $\psi^i \in C^{2, \alpha}$. 
By $ \mathscr{L}^0$ we denote the interior of the set $ \mathscr{L}$, and
by pointwise $C^{2,\alpha}$
regularity we mean uniform approximation with a second order polynomial with the speed $r^{2+\alpha}$.

\par
The idea of the proof is the same as in deriving the optimal regularity for the no-sign obstacle problem
in  \cite{ALS13}.
The proof is based on the $BMO$-estimates for $D^2 u^1_0$ and $D^2 u^2_0$ following from 
the estimate  \eqref{bddlap}. At the point $x_0$, we consider $r^{2+\alpha}$-th order rescalings
of $ u^i_0$ denoted by $v^i_r$, 
and show that these are uniformly bounded in $W^{2,2}(B_1)$. Then, looking at the corresponding system 
for $ (v^1_r, v^2_r)$,
we conclude that the rescalings
are uniformly bounded in the ball $B_1$.

\par 
In the end we justify our assumption $0\in \partial \mathscr{L}^0 $ with a counterexample:
We consider a particular system in $\mathbb{ R }^2$,
where the zero loop set $\mathscr{L}= \{0\}$, then we find an explicit solution,
that is not $C^{1,1}$.

\par
The paper is structured as follows: In Section 2 we provide some background material.
In Section 3 we use the penalization method to derive the existence of strong solutions, and
observe that these are  actually minimal solutions.
The main results are presented in the last section, where we prove that the minimal solution is locally
$C^{1,1}$ if  the zero loop set is the closure of its interior, and provide a
counterexample to $C^{1,1}$-regularity when $\psi^1+\psi^2 $ has an isolated zero.

\begin{subsection}{Acknowledgements}
 I would like to thank Prof John Andersson, Prof Diogo Gomes and Prof Henrik Shahgholian
 for several fruitful discussions. In particular, I would like to thank John Andersson
 for the idea of the proof of $C^{2, \alpha}$-regularity at the meeting points.
\end{subsection}

\begin{section}{Background material}
In this section we state some known results, which we use in our discussion, without 
giving any proofs.

\begin{subsection}{Optimal switching problem}
Let $\Omega\subset \mathbb{R}^n$ be a bounded domain with a smooth boundary.
We consider an agent that can be anywhere in $\Omega$ and in one of
a finite number $m$ of states. 
For every $1\leq i\leq m $, the agent moves in $\Omega$ according to a diffusion
\begin{equation*}
dx=b_i(x) dt +\sigma_i(x) dW_t,
\end{equation*}
where $W_t$ is a Brownian motion in a suitable probability space, $b_i:\Omega\to \mathbb{R}^n$ and
$\sigma_i: \Omega\to \mathbb{R}^{n\times m}$ are smooth functions. The generator of the diffusions
is denoted by
$ L^iv=\frac{1}{2}  \sigma_i \sigma^T_i:D^2 v+b_i \cdot Dv $.

\par
The agent can switch from any diffusion mode to another. 
At every instant $t $ the agent pays a running cost $f^{i(t)}(x)$, depending on the present state $i(t)$
and position $x$. Additionally, when changing state $i$ to state $j$ he incurs in a switching cost $-\psi^{ij}(x)$. 
Finally, when the diffusion reaches the boundary and the agent is in state $i$,
the process is stopped and a cost $-g^i(x)$ is incurred.
As it is traditional in optimal switching setting, we consider the problem of maximizing a
certain profit (the negative of the cost) functional
\begin{align*}
u_i(x)=&\max_{i(t), i(0)=i}
E \Big[\int_0^{T_{\partial \Omega}} f_{i(t)}(x(t))dt \\
&-\sum_{t\leq T_{\partial \Omega}} \psi_{i(t^-),i(t^+)}(x(t))+
g_{i(T_{\partial \Omega})}(x(T_{\partial \Omega})\Big], 
\end{align*}
where 
$T_{\partial \Omega}$ denotes the exit time  of $\Omega$. Additionally, 
the convention  
$\psi_{ii}=0$ is assumed.

\par
As it has been discussed in the literature  \cite{LB83}, \cite{BL86}, the corresponding value function $u^i$
solves the following
system: 
\begin{equation} \label{ms}
\min_{i} (-L^i u^i+f^i, \min_j (u^i-u^j+\psi^{ij}))=0. 
\end{equation}
with boundary conditions $u^i=g^i \textrm{ on } \partial \Omega$.

\par
For the optimal switching problem to be well defined, we need to impose the
nonnegative loop condition: Let 
 $i_{0},i_{1},\hdots , i_{l}=i_{0}$ be any loop of length $l$, i.e. including $l$ number of states.
 Assume that $ (u^1, u^2,...,u^m)$ is a solution to system \eqref{ms}, then 
 $ u^i-u^j+\psi^{ij} \geq 0 $ for any $ i,j \in \{ 1,2,..., m \}$, then after summing the equations
 over the loop, we get
\begin{equation*}
\sum_{j=1}^{l}\psi _{i_{j-1},i_{j}}\geq 0.
\end{equation*}
This condition is a 
necessary assumption for the existence of a solution to (3), and 
it prevents the agent from making arbitrary gains by looping.

\par
In this paper we consider a system, arising in a model optimal switching problem with only two states.
\end{subsection}

\begin{subsection}{The Poisson equation, Calderon-Zygmund estimates}

\par
We start by recalling the definition of the H\"{o}lder space $C^{k,\gamma}$.
Let us denote the continuity norm 
\begin{equation*}
 \lVert u \rVert _{C(\overline{ \Omega})}= \sup_{x\in \Omega }\lvert u(x)\rvert,
\end{equation*}
and the H\"{o}lder seminorm 
\begin{equation*}
 [ u ] _{C^{0,\gamma}(\overline{ \Omega })}= \sup_{x,y\in \Omega, x \neq y}
 \frac{ \lvert u(x)-u(y)\rvert}{ \lvert x-y \rvert^\gamma}.
\end{equation*}

\begin{definition}
 The H\"{o}lder space $C^{k,\gamma}(\Omega)$ consists of all functions $u\in C^k(\overline{\Omega})$
 such that 
\begin{equation*}
  \Arrowvert u \rVert_{C^{k,\gamma}(\Omega)}:= \sum_{\lvert \alpha \rvert \leq k}
   \Arrowvert D^\alpha u \rVert_{C(\overline{ \Omega })}+
   \sum_{\lvert \alpha \rvert = k}
   [ D^\alpha u ]_{C^{0,\gamma}(\overline{ \Omega })} < \infty.
\end{equation*}
\end{definition}

The next theorem states the known regularity of the solutions to the Poisson equation $\Delta u =f$,
under the assumption that $f$ is H\"{o}lder continuous, and can be found in the book
\cite{GT98} .

\begin{theorem}
 Assume that $f\in C^{\gamma} $, then there exists a classical solution to the 
 Poisson equation 
 \begin{equation*}
  \Delta u= f  \textrm{ in } \Omega.
 \end{equation*}
Moreover, the solution is locally $ C^{2,\gamma}(\Omega)$, and for every $ \Omega' \Subset \Omega$
 \begin{equation*}
\Arrowvert u \rVert_{C^{2,\gamma}(\Omega')} \leq
C_{n,\gamma}(\Omega') \left( \lVert u \rVert_{C(\overline{ \Omega})} + 
\Arrowvert f \rVert_{C^{0,\gamma}(\Omega)} \right),
 \end{equation*}
 where the constant $C_{n,\gamma}(\Omega') $ depends on $\textrm{diam} \Omega'$ and
 $ \textrm{dist}(\Omega', \partial \Omega) $.
\end{theorem}

\par
Next let us recall the definition of $BMO $ spaces, and then state the Calderon-Zygmund estimates
for the Poisson equation $\Delta u =f$, when $f\in L^p$, $1<p \leq \infty$.

\begin{definition}
 We say that a function $u\in L^2(\Omega)$ is in $BMO(\Omega)$ if
 \begin{equation*}
  \lVert f \rVert^2_{BMO(\Omega)}:=
  \sup_{x\in U, r>0} \frac{1}{r^n} \int_{B_r(x) \cap \Omega } \lvert f(y)-(f)_{r,x}\rvert^2 dy
  +\lVert f \rVert^2_{L^2(\Omega)}<\infty,
 \end{equation*}
where $(f)_{r,x}$ is the average of $f$ in $ B_r(x) \cap \Omega$.
\end{definition}

\par
The proofs of the following results can be found in \cite{GT98} when $p<\infty$
and in \cite{S93} when $p=\infty$.

\begin{theorem}
 Consider the equation 
 \begin{equation*}
  \Delta u= f  \textrm{ in } B_{2R}.
 \end{equation*}
 If $f \in L^p(B_{2R})$ for $ 1< p <\infty $, then the solution $u\in W^{2,p}(B_R)$, and
 \begin{equation*}
  \lVert D^2 u\rVert_{L^p(B_R)}\leq C_{p,n} \left( \lVert f \rVert_{L^p(B_{2R})} 
  +  \lVert u\rVert_{L^1(B_{2R})} \right)
 \end{equation*}
 If $f \in L^\infty (B_{2R})$, then in general $ u \notin W^{2,\infty}(B_R)$, but
\begin{equation*}
  \lVert D^2 u\rVert_{BMO(B_R)}\leq C_{\infty,n} \left( \lVert f \rVert_{L^\infty(B_{2R})} 
  +  \lVert u\rVert_{L^1(B_{2R})} \right),
 \end{equation*}
here $ C_{p,n}$, $ C_{\infty,n} $ are dimensional constants.
\end{theorem}

\end{subsection}

\begin{subsection}{The obstacle problem} 

In this section we state the regularity of the solution to the following obstacle problem,
\begin{equation*}
 \min (-\Delta u+f, u-\psi)= 0 \textrm{ in } \Omega
\end{equation*}
with boundary conditions $u-g\in W_0 ^{1,2}(\Omega)$.

\par
Here we will omit the variational formulation of the problem, the first
regularity results and will state the $ C^{1,1}$-regularity of the solutions 
referring to the book \cite{PSU12}.

\par
In order to be consistent with the assumptions in our paper, we will 
assume that $f\in C^\alpha$ and the obstacle  $\psi \in C^{2,\alpha}$,
although these assumptions can be weekened.

\begin{theorem}
Assume that $f\in C^\alpha$ and $\psi \in C^{2,\alpha}$, and 
 $ u $ solves the obstacle problem
 \begin{equation*}
 \min (-\Delta u+f, u-\psi)= 0 \textrm{  a.e. in } \Omega.
\end{equation*}
 Then $ u \in C^{1,1}(\Omega')$ for every $ \Omega' \Subset \Omega$, and
 \begin{equation*}
  \lVert u \rVert_{C^{1,1}(\Omega')}\leq C\left( \lVert u \rVert_{L^\infty(\Omega)}+
  \lVert f \rVert_{C^{0,\alpha}(\Omega)}+ \lVert \psi \rVert_{C^{2,\alpha}(\Omega)}\right),
 \end{equation*}
where the constant $C$ depends on the dimension and on the 
subset $ \Omega' \Subset \Omega$.
 
\end{theorem}

\end{subsection}

\end{section}

\section{Existence of $C^{1,\alpha}$ solutions}

\par
We consider the system \eqref{system2} with boundary conditions $ u^i=g^i \textrm{ on } \partial \Omega $,
$g^i \in C^2$.
Then we also need to impose the following compatibility condition on the boundary data:
\begin{equation} \label{bc}
 g^1-g^2+\psi^1 \geq 0, \textrm{ and }  g^2-g^1+\psi^2 \geq 0 \textrm{ on } \partial \Omega.
\end{equation}
Clearly, without the compatibility conditions, there are no solutions to \eqref{system2}
achieving the boundary data.

\par
We are interested in deriving $C^{1,1}$-regularity for the solutions to our system,
which is the best regularity one can expect.
Throughout our discussion we will assume that 
\begin{equation} \label{costass}
 f^1, f^2 \in  C^\alpha (\Omega), \textrm{ and } \psi^1, \psi^2 \in  C^{2,\alpha} (\Omega),
\end{equation}
for some $0 <\alpha <1$.
These are natural assumptions, since $f$ being bounded or continuous, 
is not enough for its Newtonian potential to be $C^{1,1}$. 
We also  provide a one-dimensional counterexample to the existence of solutions in case
the switching costs are not smooth.

\begin{example}[Diogo Gomes]
Consider the following system in the interval $(-1,1)$ with zero Dirichlet boundary conditions,
\begin{equation*}  
\begin{cases}
\min\left(-(u_{1})_{xx}, u_1-u_2+(1-|x|) \cos\left( \frac \pi
{1-|x|}\right)\right)=0 , \\
\min\left(-(u_2)_{xx}, u_2-u_1+(1-|x|)(1- \cos\left( \frac \pi
{1-|x|}\right)\right)=0.
\end{cases}
\end{equation*}
 Then the value function
of the corresponding optimal control problem is not finite.
\end{example}

\begin{proof}
In our example the running costs are identically zero, the switching costs
satisfy the nonnegative loop assumption $ \psi_{1}(x)+\psi_{2}(x)>0 $ in $ (-1,1)$, and the
compatibility condition on the boundary  $ \psi^{1}(\pm1)=\psi^{2}(\pm1)=0 $.

\par
The example illustrates that when the switching costs are not smooth, then the negative values give infinity
growth to the value function of the corresponding optimal control problem.
In order to show this, we choose optimal controls $ i(t)$ as follows: the switching occurs at times $ t_{k} $ where
$ \frac \pi {1-|x(t_{k})|}=\pi k $: When $ \frac \pi {1-|x(t_{k})|}=\pi k = \pi (2n+1) $, $n\in \mathbb{N}_0$ ,
we switch from
regime 1 to regime 2 gaining $ \frac 1 {2n+1} $ and for the values $ \frac \pi {1-|x(t_{k})|}=\pi k = 2\pi n $
we switch back from regime 2 to 1 paying zero cost, and so
\begin{equation*}
u_i(x)\geq - \sum_{0\leq t\leq T_{\partial \Omega}} \psi_{i(t_k),
i(t_{k+1})}(x(t)) = \sum \frac{1}{2n+1}.
\end{equation*}
Then the conclusion follows from the divergence of harmonic series.
 
\end{proof}

\subsection{Penalization method}

In this section we approximate the system \eqref{system2} with a smooth penalized system.
Let us take any smooth nonpositive function $\beta : \mathbb{R} \rightarrow (-\infty,0]$,
such that 
\begin{align*}
 \beta (s)=0 \textrm{ for } s\geq 0, \\
 \beta (s)<0 \textrm{ for } s<0 \textrm{ and } \\
 0<\beta'(s)\leq 1 \textrm{ for } s <0, \\
 \lim_{s \rightarrow -\infty}\beta (s)=-\infty
\end{align*}
Next we consider the following penalization function $\beta_\varepsilon (s)= \beta (s/ \varepsilon)$, 
for  $ s\in \mathbb{R}, \varepsilon >0$, and the corresponding penalized system
\begin{equation} \label{pensystem}
\begin{cases}
-\Delta u^1_{\varepsilon}+f^1+\beta_\varepsilon (u^1_\varepsilon-u^2_\varepsilon+\psi^1)=0 \\
-\Delta u^2_{\varepsilon}+f^2+\beta_\varepsilon (u^2_\varepsilon-u^1_\varepsilon+\psi^2)=0 ,
\end{cases}
\end{equation}
with boundary conditions $ u^i_\varepsilon=g^i \textrm{ on } \partial\Omega $.

\par
For $\varepsilon >0$ fixed, the penalized system \eqref{pensystem}
can be solved by several methods. In the paper 
\cite{EF79} the authors use nonlinear functional analysis methods in order to derive the existence 
of classical solutions, that is $u^i_\varepsilon \in C^2(\Omega)$, assuming that the switching costs are positive constants. 
The proof is rather technical, however it works line for line in our case with variable switching costs,
therefore we omit it.

\begin{lemma} \label{penlemma}
Under the assumptions \eqref{bc} and \eqref{costass} the solutions to the penalized system 
\eqref{pensystem}, $u^i_{\varepsilon}$ satisfy the
following estimates for every $\varepsilon>0$
\begin{enumerate}

\item [ i.) ]
 \begin{equation*}  
 - \max_i  \Arrowvert f^i \rVert_{L^{\infty}} \leq -\Delta u^i_{\varepsilon} 
 \leq \max_i \Arrowvert \Delta \psi^i \rVert_{L^{\infty}}+
 3 \max_i \Arrowvert f^i \rVert_{L^{\infty}}.
 \end{equation*}
 
 \item[ ii.) ]
 \begin{equation*}
 u^1_\varepsilon-u^2_\varepsilon+\psi^1 \geq -C \varepsilon 
 \textrm{ and }
 u^2_\varepsilon-u^1_\varepsilon+\psi^2 \geq -C \varepsilon 
\end{equation*}
\end{enumerate}
In $ ii.) $ the constant $C>0$ depends only on the given data and can 
be computed explicitly in terms of $\beta $.
\end{lemma}

\begin{proof}
 For our convenience, let us denote $\theta^1_\varepsilon=u^1_\varepsilon-u^2_\varepsilon+\psi^1$ and 
  $\theta^2_\varepsilon=u^2_\varepsilon-u^1_\varepsilon+\psi^2$, and observe that $\theta^1_\varepsilon $
  and $\theta^2_\varepsilon $ cannot be 
  negative at the same time according to
  the nonnegative loop assumption. 
  
  \par
  Now let us fix $\varepsilon >0$, 
  and consider the function $\beta_\varepsilon( \theta^i_\varepsilon(x))$, $x \in \Omega$. 
  It is bounded from above by $0$, our aim is to prove that $\beta_\varepsilon( \theta^i_\varepsilon(x))$ is
  bounded from below.
  Let $x_0=x_0(\varepsilon)$ be a point of minimum for the function $\beta_\varepsilon( \theta^1_\varepsilon(x))$,
  moreover without loss of generality, we may assume that 
  \begin{equation*}
   \min_{i=1,2;x\in \overline{\Omega}} \beta_\varepsilon( \theta^i_\varepsilon(x))=
   \beta_\varepsilon( \theta^1_\varepsilon(x_0))<0.
  \end{equation*}
If $x_0 \in \partial \Omega$, then $ \beta_\varepsilon( \theta^1_\varepsilon(x_0))=0$ 
according to \eqref{bc}. Therefore $ x_0 \in \Omega $ is an interior point, and $ \beta_\varepsilon( \theta^1_\varepsilon(x_0))<0$.
Then $ \theta^1_\varepsilon(x_0) <0$, and since $\theta^1_\varepsilon+ \theta^2_\varepsilon \geq 0$, we get
$ \theta^2_\varepsilon(x_0) \geq 0 $ consequently $ \beta_\varepsilon( \theta^2_\varepsilon(x_0))=0$.
Since $\beta_\varepsilon$ is nondecreasing and $\beta_\varepsilon(t)<0$ if and only if $t<0$, we 
get that 
\begin{equation*}
   \min_{i=1,2;x\in \overline{\Omega}}  \theta^i_\varepsilon(x)=
   \theta^1_\varepsilon(x_0).
 \end{equation*}
 This implies that $\theta ^1_\varepsilon= u^1_\varepsilon-u^2_\varepsilon+\psi^1 $ 
 achieves its minimum at an interior point $x_0$,
 hence $\Delta u^1_\varepsilon-\Delta u^2_\varepsilon+\Delta \psi^1\geq 0$
at $x_0$. The last inequality together with $-\Delta u^2_\varepsilon(x_0)+f^2(x_0)=0$ shows that 
\begin{equation*}
 \beta_\varepsilon( \theta^1_\varepsilon(x_0))= \Delta u^1_{\varepsilon}(x_0)-f^1(x_0)=
\end{equation*}
\begin{equation*}
 \Delta u^1_{\varepsilon}(x_0)-\Delta u^2_{\varepsilon}(x_0) +f^2(x_0)-f^1(x_0)
 \geq -\Delta \psi^1 (x_0)+f^2(x_0)-f^1(x_0).
\end{equation*}
The estimate above is true for any $\varepsilon >0$, and therefore it proves the right inequality in $ i.) $. 
The left inequality in $ i.) $ is a direct consequence of $ -\beta_\varepsilon \geq 0$.

\par 
In order to prove  $ ii.) $, we
recall that $\lim_{s\rightarrow -\infty} \beta(s)=-\infty$, and
$\beta_\varepsilon(s)=\beta(s/\varepsilon)$, hence $\beta_\varepsilon(\theta^i_\varepsilon)$ is
bounded imples that $ \frac{\theta^i_\varepsilon}{\varepsilon}$ is uniformly bounded from below 
by a negative constant $-C\leq 0$. This finishes the proof of point $ ii.) $ in our lemma.
\end{proof}

\par
Using the Sobolev embedding theorem and Calderon-Zygmund estimates, we can conclude that
the functions $u^i_\varepsilon$ are uniformly bounded in $W^{2,p}$ for every $1<p<\infty$.
Therefore through a subsequence $u^i_{\varepsilon}$ converges to a function $u^i_0$
locally weakly in $W^{2,p}$ and strongly in $C^{1,\gamma}$ for every $0<\gamma<1$. 

\par
Now we proceed to prove the existence of solutions to system \eqref{system2}.

\begin{proposition}
Let $ (u^1_0,u^2_0)= \lim_{\varepsilon \rightarrow 0} (u^1_\varepsilon,u^2_\varepsilon)$ 
through a subsequence weakly in $W^{2,p}$ and strongly in $C^{1,\gamma}$.
 Then $(u^1_0,u^2_0)$ solves the following system 
 \begin{equation} \label{system3}
\begin{cases}
\min (-\Delta u^1_0+f^1, u^1_0-u^2_0+\psi^1)=0, \\
\min (-\Delta u^2_0+f^2, u^2_0-u^1_0+\psi^2)=0,  \\
\min (-\Delta u^1_0+f^1, -\Delta u^2_0+f^2)=0
\end{cases}
\end{equation}
in a strong sense, i.e.
$ u^i-u^j+\psi^i \geq 0$ 
and if we have a strict inequality at some point then $u^i$ satisfies $ \Delta u^i_0=f^i$ in a
neighborhood of that point, and  $ -\Delta u^i_0+f^i \geq 0 \textrm{ a.e. } $.
\end{proposition}

\begin{proof}
The property $ ii.) $ in Lemma 1, together with the strong convergence in $C^{1,\gamma}$ shows that
$ u^1_0-u^2_0+\psi^1\geq 0$ and $u^2_0-u^1_0+\psi^2 \geq 0$. If $ u^1(x_0)-u^2(x_0)+\psi^1(x_0) >0$,
then the strict inequality
$u^1_\varepsilon-u^2_\varepsilon+\psi^1>0$   holds in a small ball $B_r(x_0)$, centered at $x_0$
for $\varepsilon>0$ small enough. Then it follows that  $ -\Delta u^1_\varepsilon+f^1= 0$ in $B_r(x_0)$, 
and we know that $ \Arrowvert  \Delta u^1_\varepsilon \rVert_{L^\infty}$ is uniformly bounded, 
therefore through a subsequence, 
$ \Delta u^1_\varepsilon \rightarrow \Delta u^1_0 $  a.e. as $\varepsilon \rightarrow 0$, 
consequently
  $ -\Delta u^1+f^1= 0$  a.e. in $B_r(x_0)$. Moreover, since $f^1 \in C^\alpha$, we that $u^1$ is
  a classical solution to
  $ -\Delta u^1+f^1= 0$ in the ball $B_r(x_0)$.

\par
The solutions of the penalized system satisfy the equation
\begin{equation*}
 \min(-\Delta u^1_\varepsilon+f^1, -\Delta u^2_\varepsilon+f^2)=0.
\end{equation*}
After passing to a limit through a subsequence, we get the following
\begin{equation*}
 \min(-\Delta u^1_0+f^1, -\Delta u^2_0+f^2)=0 \textrm{ a.e.}.
\end{equation*}
\end{proof}

\par
Proposition 1 shows that there exists $(u^1_0, u^2_0), u^i_0 \in W^{2,p}, \forall p <\infty$ solving 
\eqref{system2} in a strong sense. According to Lemma 1, $u^i_0$ has the following property

\begin{equation}\label{keyest}
\Arrowvert \Delta u^i_0 \rVert_{L^\infty} \leq
\max_i \Arrowvert \Delta \psi^i \rVert_{L^{\infty}}+
 3 \max_i \Arrowvert f^i \rVert_{L^{\infty}},
\end{equation}
which will be relevant for deriving further regularity of solutions.

\par
 Furthermore, Proposition 1 tells us that the solution we get via the penalization method, 
 solves an extra equation,
 which turns out to be very important in the discussion of the uniqueness.

\subsection{Uniqueness}
It has been shown in the paper \cite{LB83} that if there are no zero loops, then the
solution to the system \eqref{system2} is unique. Here we give a counterexample showing
that the uniqueness does not hold in case 
there are zero loops.

\begin{example}[Diogo Gomes] 
The following system
\begin{equation} \label{exun}
\begin{cases}
\min (-\Delta u^1-M, u^1-u^2+\psi)=0 \\
\min (-\Delta u^2+M, u^2-u^1-\psi)=0,
\end{cases}
\end{equation}
with given boundary conditions $u^i=g^i, g^1-g^2+\psi =0 $ on $\partial \Omega $,
admits infinitely many solutions, provided $ 2 M > \Arrowvert \Delta \psi \rVert_{L^{\infty}} $.
\par
Moreover, \eqref{exun} admits solutions $u^1, u^2 \notin C^{1,1}$.
\end{example}

\begin{proof}
Let $ (u^1, u^2)$ be a solution to the system \eqref{exun}.
Since both $u^1-u^2+\psi \geq 0$  and $u^2-u^1-\psi \geq 0$, it follows that
 $u^1-u^2+\psi \equiv 0$, therefore $ -\Delta u^1= -\Delta u^2+ \Delta \psi$.
 
 \par
 Now let us take any $u^1 \in W^{2,p}$, $p>n$,  $u^1=g^1$ on $\partial \Omega $, such that $ -\Delta u^1-M \geq 0$  a.e..
 Then the function $u^2=u^1+\psi $ satisfies the boundary conditions
$ u^2=g^2$ on $\partial \Omega $, and  $ -\Delta u^2+M \geq 0$ a.e. 
 since $ 2 M > \Arrowvert \Delta \psi \rVert_{L^{\infty}} $. 
Thus we get infinitely many solutions of the form $(u^1, u^1+\psi)$, which may not be $C^{1,1}$.
\end{proof}

\par
We observe that if the zero loop set is empty, then the equation 
$  \min(-\Delta u^1+f^1, -\Delta u^2+f^2)=0$ is satisfied automatically. 
Under the nonnegative loop assumption, we saw that there exists a 
solution to system \eqref{system2} also solving system \eqref{system3}. Next we show that 
the system \eqref{system3} has a unique solution, which is actually the minimal solution to (1).

\begin{proposition}
 The system \eqref{system3} has a unique solution $(u^1_0,u^2_0)$ in $W^{2,p}$ for every $p<\infty$.
\end{proposition}

\begin{proof}
Let us assume that $(u^1,u^2)$ is a solution to system \eqref{system3}, then 
the difference $ U= u^1-u^2 $ solves the following double-obstacle
problem in $\Omega $:
\begin{equation*}
\begin{cases}
-\Delta U +f^1-f^2 \leq 0 \textrm{ a.e.   if  } U >-\psi^1 \\
-\Delta U +f^1-f^2 \geq 0 \textrm{ a.e.   if  } U <\psi^2  \\
-\psi^1 \leq U \leq \psi^2,
\end{cases}
\end{equation*}
with boundary conditions $ U= g^1-g^2 $
on $\partial \Omega $.

\par
It is well-known that the solution to the double-obstacle problem with given boundary data 
is unique in $W^{2,p}$. Indeed, let $V $ be another solution, then 
without loss of generality, we may assume that 
$ \max_{x\in \Omega } (U-V)= U(x_0)-V(x_0)>0$.
Then in a small ball $B_r(x_0)$, one has $ U-V>0$, and $U-V$
 has a maximum at $x_0$. The inequality $ U>V \geq -\psi^1$ imples that
 $U > -\psi^1$ in $B_r(x_0)$, hence $ -\Delta U -f^1+f^2 \leq 0$.
 Similarly, $V < \psi^2 $ in $B_r(x_0)$ and therefore $ -\Delta V -f^1+f^2 \geq 0$.
 After combining the inequalities $ -\Delta U -f^1+f^2 \leq 0$ and
 $ -\Delta V -f^1+f^2 \geq 0$, we see that $U-V$ is a subharmonic function in the 
 ball $B_r(x_0)$. Recalling that $U-V $ has a maximum at an interior point $x_0$, we get a 
 contradiction to the maximum principle for subharmonic functions.

\par
Now let us assume that $(v^1,v^2)$ 
is another solution to system \eqref{system3}, then $ u^1-u^2 \equiv v^1- v^2 $ in $\Omega$.
Denote $ h= u^1-v^1 \equiv u^2- v^2$ in $\Omega$, then $h=0$ on $\partial \Omega$.

\par
Now let us plugg-in $ v^1= u^1-h $ and $ v^2= u^2-h $   to the equation
 \begin{align*}
 0=\min(-\Delta v^1+f^1, -\Delta v^2+f^2)=\\
 \min(-\Delta u^1+f^1, -\Delta u^2+f^2)+ \Delta h= \Delta h \textrm{        a.e. }.
 \end{align*}
 Then it follows that $\Delta h = 0  \textrm{    a.e. } $ in $\Omega$, $h\in W^{2,p}(\Omega)$,
 for every $ 1 <p<\infty$, hence
 $ h $ is a harmonic function.  Then the difference $u^i-v^i$ is a harmonic function in 
 $\Omega $, vanishing on the boundary, therefore
 $u^i-v^i \equiv 0$, according to the maximum principle for harmonic functions.

\end{proof}

\begin{corollary}
The solution to the system \eqref{system3} is the minimal solution to system \eqref{system2}, that is
 if $(v^1, v^2)$ solves \eqref{system2}, then $u^1_0\leq v^1$ and $u^2_0 \leq v^2$.
\end{corollary}
\begin{proof}
 Assume $(v^1, v^2)$ solves \eqref{system2}
with given boundary conditions, and let 
$ \omega = \min (-\Delta v^1+f^1, -\Delta v^2+f^2)$ , then $ \omega \geq 0 \textrm{ a.e.}$, $ \omega \in L^{\infty}$.
Let $h$ be the solution to  $\Delta h = \omega $ in $\Omega $ with zero Dirichlet
boundary conditions on $\partial \Omega$. Then according to the weak maximum principle
for subharmonic functions, we get that $h \leq 0 $ in $\Omega$. 

\par
Now we note that 
the pair $(v^1+h, v^2+h)$ solves the system \eqref{system3} with the same boundary conditions as 
$(u^1_0, u^2_0)$, hence
$v^1+h=u^1_0 \textrm{ and }  v^2+h=u^2_0 $. Then the minimality of $(u^1_0, u^2_0)$
follows from nonpositivity of $h$.
\end{proof}

From now on we will be interested in studying the regularity for the minimal solutions.
As the Example 2 shows, there is no hope to get $C^{1,1}$-regularity for non-minimal solutions.

\section{Optimal regularity of the solutions}

In this section we prove that the solution to the system \eqref{system3} is locally $C^{1,1}$,
 if $\mathscr{L}= \overline{\mathscr{L}^0}$. In particular 
we study the regularity of the solutions on $ \partial \mathscr{L}$, the boundary of the zero-loop set.

\par
Before proceeding to the discussion of $C^{1,1}$-regularity, let us rewrite our system in a more 
convenient way. We have assumed that $ f^1, f^2 \in  C^\alpha$, therefore there exist
$ v^1, v^2 \in C^{2,\alpha}_{loc}$ solving the Poisson equation $\Delta v^i= f^i$ in $\Omega$.
Recall that $ (u^1_0, u^2_0)$ is the solution to system \eqref{system3}, and define 
$ u^i= u^i_0-v^i $, then $u^i_0$ is as regular as  $u^i$ up to $C^{2,\alpha}$, and
 $ (u^1, u^2 )$ solves the following system 
\begin{equation} \label{system}
\begin{cases}
\min (-\Delta u^1, u^1-u^2+\varphi^1)=0,\\
\min (-\Delta u^2, u^2-u^1+\varphi^2)=0,\\
\min (-\Delta u^1,-\Delta u^2)=0.
\end{cases}
\end{equation}
Here $ \varphi^1= v^1-v^2+\psi^1$ 
and $ \varphi^2= v^2-v^1+\psi^2$ are the new switching cost functions preserving
the loop condition $\varphi^1+\varphi^2 \equiv \psi^1+\psi^2$, and
$ \varphi^1, \varphi^2 \in C^{2,\alpha}_{loc}$. 

\par
From now on we will be focused on studying the regularity  of $(u^1, u^2)$ solving
the system \eqref{system}.

\par
We define the open set $ \Omega_1:= \cup B_r(x_0, u^1)$, where the union 
is taken over the balls $ B_r(x_0, u^1)$, such that $ -\Delta u^1 > 0$ a.e. 
in $B_r(x_0, u^1)$. Similarly we define the set $\Omega_2$ corresponding to the function 
$u^2$, and let $\Omega_{12}= \Omega \setminus \overline{ \Omega_1 \cup \Omega_2}$.
Then $\Omega_1$, $\Omega_2 $ and $\Omega_{12}$
are disjoint open sets, and since $\varphi^1, \varphi^2 \in C^{2,\alpha}$,
\begin{align*}
 -\Delta u^1= \Delta \varphi^1 >0, -\Delta u^2= 0  \textrm{  in  } \Omega_1, \\
 -\Delta u^2= \Delta \varphi^2 >0, -\Delta u^1= 0  \textrm{  in  } \Omega_2,  \\
 -\Delta u^1=0, -\Delta u^2=0 \textrm{  in  } \Omega_{12}.
\end{align*}

\par
In the set $\Omega \setminus \overline{\Omega_1}$ we get 
$-\Delta u^1=0$, and the function $u^2$ solves the
 obstacle problem $\min(-\Delta u^2, u^2-u^1+\varphi^2)= 0$, with a $C^{2, \alpha}$ obstacle
 $ u^1-\varphi^2$. Therefore $u^2$ is locally $C^{1,1}$ in $  \Omega \setminus \overline{\Omega_1} $.
 Similarly we get that $u^2$ is locally $C^{1,1}$ in $  \Omega \setminus \overline{\Omega_2} $.

 \par
 Next we need to study the regularity of the solution in a neighborhood of the set 
 $\partial \Omega_1 \cap \partial \Omega_2 \cap \Omega$. Let us note that it is contained in
 the zero loop set, $\partial \Omega_1 \cap \partial \Omega_2 \subset \mathscr{ L}$, since 
 $ u^1-u^2+\varphi^1 \equiv 0 \textrm{ in } \Omega_1$ and
 $ u^2-u^1+\varphi^2 \equiv 0 \textrm{ in } \Omega_2$.
 In the interior of the zero loop set the system \eqref{system} reduces to the equation 
 \begin{equation} \label{Leq}
 -\Delta u^1 = (\Delta \varphi^1)^+, u^2=u^1+\varphi^1 \textrm{ in } \mathscr{ L }^0.
 \end{equation}
 From the classical theory, solutions to the equation \eqref{Leq} are locally $C^{2, \alpha}$
 if $\Delta \varphi^1  \in C^{ \alpha}$. So in a neighborhood of the points 
 $x \in \partial \Omega_1 \cap \partial \Omega_2 \textrm{ and } x \in \mathscr{ L}^0 $,
 the solution is $C^{2, \alpha} $.
 
 \par
 It remains to study the regularity of $(u^1, u^2)$ at the points
 $ x_0 \in \partial \Omega_1 \cap \partial \Omega_2 \cap \partial \mathscr{ L}$, called a 
 ''meeting'' point. In this section we show that $u^1$ and $u^2$ are actually $C^{2,\alpha}$-regular at such points.
 
 \par
 For simplicity, let us study the system locally in the unit ball $B_1$, assuming that
 $0\in \partial \mathscr{L} \cap \partial \Omega_1 \cap \partial \Omega_2$. 
 We can always come to such a situation with a change of variables.

 \subsection{Blow-up procedure}
Assume that $(u^1, u^2)$ solves system \eqref{system} in the unit ball $B_1$, and
$0\in \partial \mathscr{L} = \partial \mathscr{L}^0 $ is a meeting point, and let us study the regularity of the 
functions $u^1$ and $u^2$ at $0$.

\begin{definition} \label{appol}
For a function $u \in W^{2,2}$, define $\Pi(u(x),r)=p_r(x) $, where
$p_r(x)= x \cdot A_r\cdot x+b_r\cdot x+c_r$ is a second order polynomial with the matrix $A_r$, vector $b_r$ 
and scalar $c_r$ minimizing the following expression
\begin{equation*}
 \int_{B_r}\left( |D^2u-2A_r|^2+|\nabla u -b_r|^2 +|u-c_r|^2 \right)dx.
\end{equation*}
\end{definition}

Then $ \Pi (u(rx),1)=p_r(rx)$, and
it is easy to see that 
\begin{equation*}
 p_r(x)=\frac{1}{2} x \cdot (D^2u)_r \cdot x+ (\nabla u)_r\cdot x + (u)_r,
\end{equation*}
where $ (u)_r:=(u)_{r,0}$, and $(u)_{r,x_0}$ is the average of $ u $ over the ball
$B_r(x_0)= \{ x \in \mathbb{R}^n \mid \lvert x-x_0 \rvert <r \}$,
\begin{equation*}
(u)_{r,x_0}=\frac{1}{\lvert B_r\rvert} \int_{B_r(x_0)}u.
\end{equation*}

\par
Now let $0 <r<1 $ be a real number, and define 
\begin{equation*}
 v^i_r(x)= \frac{u^i(r x)-\Pi(u^i(r x),1)}{S(r)},
\end{equation*}
where $S(r)$ is chosen such that $\max_i \Arrowvert D^2v^i_r \rVert_{L^2(B_1)}=1$.
Our aim is to describe the rate of convergence of $S(r)$ as $r$ goes to zero.

\par
It follows immediately from our definition of $ S(r)$, and $BMO$-estimates 
that  $\frac{S(r)}{r^2}$ is uniformly bounded from above.
In order to show this, let us recall that
 $ \Arrowvert \Delta u^i \rVert_{L^\infty} \leq \max_ i \Arrowvert \Delta \varphi^i \rVert_{L^\infty}$,
hence
$D^2 u^i \in BMO $ locally, with the following estimate
\begin{equation*}
 \Arrowvert D^2 u^i \rVert_ {BMO(B_{\frac{1}{2}})}\leq C(\max_ i \Arrowvert \Delta \varphi^i \rVert_{L^\infty}
 +\Arrowvert u^i \rVert_{L^2(B_1)}).
\end{equation*}

Without loss of generality, we may assume that $\Arrowvert D^2 v^1_r \rVert_{L^2(B_1)}=1$, for a fixed $r>0$, then 
\begin{equation*}
 \frac{S(r)}{r^2}=\Arrowvert D^2 u^1(r x)-2A^1_{r} \rVert_{L^2(B_1)}.
\end{equation*}

\par
A change of variable will give us
\begin{align*}
 \left(\frac{S(r)}{r^2} \right)^2=\frac{1}{r^n} \int_{B_{r}}|D^2u^1(y)-2A^1_{r}|^2dy =\\
 \frac{1}{r^n} \int_{B_{r}}|D^2u^1(y)-(D^2u^1)_{r}|^2 dy\leq 
 C \left( \max_ i \Arrowvert \Delta \varphi^i \rVert_{L^\infty}+\Arrowvert u^1 \rVert_{L^2(B_1)}\right)^2,
\end{align*}
therefore $ \forall r<\frac{1}{2} $
\begin{equation}
 \frac{S(r)}{r^2}
  \leq C \left(\max_ i \Arrowvert \Delta \varphi^i \rVert_{L^\infty}+
\max_i \Arrowvert u^i \rVert_{L^2(B_1)} \right) :=C_0.
\end{equation}

 \par
So $S(r)$ has at least quadratic decay  as $r\rightarrow 0$.
 Next we improve the estimate, showing that actually 
 $S(r) \leq C_0 r^{2+\alpha }$ for $r>0$ small enough.

 \begin{proposition}
 Let $\varphi^1, \varphi^2 \in C^{2, \alpha}$ for some $0<\alpha<1$, and $\mathscr{L}= \overline {\mathscr{L}^0}$,
 then the function $\frac{S(r)}{r^{2+\alpha}}$ is uniformly bounded as $r$ goes to zero
 \begin{equation} \label{main}
 \frac{S(r)}{r^{2+\alpha}}  \leq C \left( \max_ i \Arrowvert \Delta \varphi^i \rVert_{L^\infty}+
\max_i \Arrowvert u^i \rVert_{L^2(B_1)} \right),
\end{equation}
where $C$ is a dimensional constant.
\end{proposition} 

\begin{proof}
Let us start with an important observation: The assumptions $ 0\in \partial \mathscr{L} \cap \partial \Omega_1 \cap
\partial \Omega_2$,
$\mathscr{L}=\{ \varphi^1+\varphi^2 =0 \}$, $\mathscr{ L}= \overline{\mathscr{ L}^0}$ and 
$ \varphi^1, \varphi^2 \in C^{2, \alpha}$ imply that $ \Delta \varphi^1(0)+\Delta \varphi^2(0)=0 $. 
On the other hand $ \Delta \varphi^1>0 $ in $\Omega_1$ and 
$ \Delta \varphi^2>0 $ in $\Omega_2$, therefore $ \Delta \varphi^1(0) =
\Delta \varphi^2(0)=0$.

\par
Next we show that  $\varphi^i \in C^{2, \alpha}$ together with $\Delta \varphi^i(0)=0$, provide the growth 
estimate \eqref{main}. The proof is based on an argument of contradiction, assume that $\frac{S(r)}{r^{2+\alpha}}$
is not bounded, 
then there exists a sequence
$r_k\rightarrow 0$ as $ k\rightarrow\infty $, such that $S(r_k)=kr_k^{2+\alpha}$ and $S(r)\leq kr^{2+\alpha}$ for
all $r \geq  r_k$. Our aim is to study the convergence of the sequence $ v^i_{k}:=v^i_{r_k}$ as $k\rightarrow\infty$.
For that we will need some basic properties of the functions $v^i_r$, where $ 0<r<1$.

\par
According to the definition of $S(r)$, $ \Arrowvert D^2 v^i_r \rVert_{L^2(B_1)}\leq1$ for $r<1$.
Then applying Poincare's inequality for  the function $\nabla v^i_r$ in the unit ball $B_1$,
we get $ \Arrowvert \nabla v^i_r \rVert_{L^2(B_1)}\leq C_n$ for every $0<r<1$, since
 $(\nabla v^i_r)_1= 0$. Next we study the average of $v^i_r$ in the unit ball
 \begin{align*}
(v^i_r)_1=\frac{(u^i)_{r}-\frac{n}{2} r^2 \cdot tr(D^2 u^i)_{r} \fint_{B_1}x_1^2 dx -
r \fint_{B_1} x \cdot (\nabla u^i)_{r}dx
-(u^i)_{r}}{S(r)} \\
= -\alpha_n \frac{r^2 \left( \Delta u^i (r x) \right)_1}{S(r)},
\end{align*}
where $\alpha_n= \frac{n}{2}\fint_{B_1}x_1^2 dx$ is a dimensional constant.
Now let us recall that $ (u^1, u^2)$ solves \eqref{system},
and therefore 
\begin{align*}
0 \leq -\Delta u^i(r x) \leq \max (0, \Delta \varphi^i (r x))\leq 
\Arrowvert \varphi^i \rVert_{C^{2,\alpha}} r^\alpha \lvert x \rvert^\alpha, 
\end{align*}
since $ \Delta \varphi^i (0) =0$, $\varphi^i \in C^{2, \alpha}$. 
Hence we get $ 0\leq (v^i_r)_1 \leq C_n \frac{r^{2+\alpha}}{S(r)} \Arrowvert \varphi^i \rVert_{C^{2,\alpha}}$,
where $C_n>0$ is a dimensional constant.
Next we apply Poincare's inequality one more time,
$ \Arrowvert v^i_r- (v^i_r)_1 \rVert_{L^2(B_1)} \leq C_n \Arrowvert \nabla v^i_r \rVert_{L^2(B_1)}$.
 Therefore  we may conclude that  
\begin{equation} \label{w22}
 \Arrowvert \nabla v^i_r \rVert_{L^{2}(B_1)}\leq C,
 \Arrowvert v^i_r \rVert_{L^{2}(B_1)}\leq C\left(1+  
 \lVert \varphi \rVert_{C^{2,\alpha}}\frac{r^{2+\alpha}}{S(r)}\right),
\end{equation}
for every $ 0<r<1$, where the constant $C>0$ depends only on the dimension.

\par
Next by using \eqref{w22} we want to estimate the  $\Arrowvert v^i_k \rVert_{W^{2,2}(B_R)} $ 
for $ R < 1/r_k $ as $k\rightarrow\infty$.
Let us start by looking at the expressions $ \lvert A^i_{2^lr_k}-A^i_{2^{l-1}r_k} \rvert$,
where $l\in \mathbb{N}$ and $r_k<1$ are chosen such that $s=r_k 2^{l-1}\leq \frac{1}{4}$.
It follows from Minkowski's inequality, that
 \begin{align*}
 \lvert A^i_{2s}-A^i_s \rvert \leq {\left( \fint _{B_1} |D^2u^i(xs)-2A^i_s|^2\right)}^{\frac{1}{2}}+
   \left(\fint _{B_1} |D^2u^i(xs)-2A^i_{2s}|^2 \right)^{\frac{1}{2}}\\
 \leq \frac{S(s)}{s^2}+2^{\frac{n}{2}} \left(\fint _{B_{\frac{1}{2}}} |D^2u^i(2xs)-2A^i_{2s}|^2 \right)^{\frac{1}{2}} 
\leq \frac{S(s)}{s^2}+2^{\frac{n}{2}}\frac{S(2s)}{4s^2}.
 \end{align*}
Hence
\begin{equation*}
 \lvert A^i_{2^lr_k}-A^i_{2^{l-1}r_k} \rvert \leq  k(1+2^{\frac{n}{2}+\alpha})(r_k 2^{l-1})^\alpha,
\end{equation*}
provided $r_k 2^{l-1} \leq \frac{1}{4}$.

\par
Now let us take any $m \in \mathbb{ N }$ such that $ 2^{m+1} r_k \leq 1$, then
\begin{align*}
  \left(\int _{B_{2^m}}|D^2 v^i_k(x)|^2dx \right)^{\frac{1}{2}}=\frac{r_k^2}{S(r_k)}
  \left(\int _{B_{2^m}}|D^2 u^i(r_k x)-2A^i_{r_k}|^2dx \right)^{\frac{1}{2}}\\  
  \leq \frac{ 2^{\frac{mn}{2}}}{k r_k^{\alpha}}
  \left(\int _{B_1}|D^2 u^i(2^m r_k x)-2A^i_{r_k}|^2dx \right)^{\frac{1}{2}} \\
   \leq \frac{ 2^{\frac{mn}{2}}}{k r_k^{\alpha}}
  \left(   \left(\int _{B_1}|D^2 u^i(2^m r_k x)-2A^i_{2^m r_k}|^2dx \right)^{\frac{1}{2}}+
  \lvert A^i_{2^mr_k}-A^i_{r_k} \rvert \right)\\
    \leq \frac{ 2^{\frac{mn}{2}}}{k r_k^{\alpha}}
    \left(  \frac{S(2^m r_k)}{(2^m r_k)^2} +\sum_{j=1}^{m}\lvert A^i_{2^jr_k}-A^i_{2^{j-1}r_k} \rvert \right)\\
\leq \frac{ 2^{\frac{mn}{2}}}{k r_k^{\alpha}} \left(k 2^{m \alpha }r_k^\alpha
+  k 2^n r_k^\alpha \sum_{j=1}^{m} 2^{\alpha (j-1)} \right) \leq
2^{n+1} 2^{m\left( \frac{n}{2}+\alpha \right)}.
 \end{align*}

 \par
For every $R<\frac{1}{2r_k}$ we can find an $m \in \mathbb{N }$ such that 
$2^{m-1}\leq R<2^m $, and then applying the estimates above, we get 
\begin{equation*}
 \int _{B_R}|D^2 v^i_k(x)|^2dx\leq C_n R^{n+2\alpha},
\end{equation*}
for every $R<\frac{1}{2r_k}$, where $C_n$ is a dimensional constant.
 Then we can also show that $ \Arrowvert \nabla v^i_{r_k} \rVert_{L^2(B_R)}$ and
 $ \Arrowvert v^i_{r_k} \rVert_{L^2(B_R)}$ are bounded by a constant depending only on $R $.
 Indeed, applying the corresponding estimates for $A^i_{r_k}$, and the first inequality in \eqref{w22}, we get
 \begin{equation*}
  \lvert b^i_{2^l r_k}- b^i_{2^{l-1}r_k}\rvert \leq C_{n,\alpha} k(r_k 2^{l})^{1+\alpha},
 \end{equation*}
and therefore
\begin{equation*}
   \lvert b^i_{R r_k}- b^i_{r_k}\rvert \leq C_{n,\alpha} k (r_k R )^{1+\alpha}.
\end{equation*}

Then the Poincare's inequality in a ball $B_R$ implies that
\begin{equation*}
 \lVert \nabla v^i_k -( \nabla v^i_k)_R \rVert_{L^2(B_R)} \leq C_n R \lVert D^2 v^i_k \rVert_{L^2(B_R)},
\end{equation*}
and
\begin{equation*}
 \lVert  v^i_k -( v^i_k)_R \rVert_{L^2(B_R)} \leq C_n R \lVert \nabla v^i_k \rVert_{L^2(B_R)},
\end{equation*}
where 
\begin{align*}
 ( \nabla v^i_k)_R= \frac{r_k}{S(r_k)} ( \nabla u^i_{r_k}-r_k A^i_{r_k} \cdot x -b^i_{r_k} )_R \\
= \frac{r_k}{S(r_k)} \left(  (\nabla u^i )_{R r_k}- b^i_{r_k}\right)
 = \frac{r_k}{k r_k^{2+\alpha}} \left(  b^i_{ R r_k}- b^i_{r_k} \right),
\end{align*}
and
\begin{align*}
 (v^i_k)_R=\frac{1}{S(r_k)} \left( c^i_{R r_k} -c^i_{r_k} +\frac{n}{2} r_k^2(\Delta u^i)_{r_k}(x_1^2)_R \right).
\end{align*}

\par
Next let us observe that the second inequality in \eqref{w22}, with the corresponding estimates for 
$A^i_r$ and $b^i_r$ imply that
\begin{equation*}
 \lvert c^i_{Rr_k}-c^i_{r_k}\rvert \leq C k(R r_k)^{2+\alpha}.
\end{equation*}

\par
Then it follows from the triangle's inequality that
\begin{equation*}
 \lVert \nabla v^i_k \rVert_{L^2(B_R)} \leq  C \left(  R^{\frac{n}{2}+1+\alpha}+
 R^{\frac{n}{2}}\frac{r_k}{k r_k^{2+\alpha}}\lvert b^i_{Rr_k}-b^i_{r_k}\rvert \right) \leq C_{n}R^{\frac{n}{2}+1+\alpha},
\end{equation*}
and also
\begin{align*}
 \lVert v^i_k\rVert_{L^2(B_R)}\leq C \left( R \lVert \nabla v^i_k \rVert_{L^2(B_R)}+
 R^{\frac{n}{2}}(v^i_k)_R \right) \leq \\
 C\left( R^{\frac{n}{2}+2+\alpha} +R^{\frac{n}{2}} \frac{\lvert c^i_{Rr_k}-c^i_{r_k}\rvert}{S(r_k)} +
  \lVert \varphi \rVert_{C^{2,\alpha}} R^{\frac{n}{2}+2} \frac{r_k^{2+\alpha}}{S(r_k)}\right)
 \leq C' R^{\frac{n}{2}+2+\alpha}.
\end{align*}

 \par
Therefore we have shown that the sequence $v^i_k$ is locally uniformly bounded in $W^{2,2}$, hence 
 through a subsequence, $v^i_k$ converges weakly in $W^{2,2}(B_R)$, and strongly in $W^{1,2}(B_R)$, denote 
 $ v^i_0= \lim_{k\rightarrow\infty}v^i_k$ for $i=1,2$.
 Then the weak convergence of the second order derivatives implies that 
 \begin{equation*}
   \int _{B_R}|D^2 v^i_0(x)|^2dx \leq \limsup_{k \rightarrow \infty} 
 \int _{B_R}|D^2 v^i_k(x)|^2dx ,
 \end{equation*}
 and therefore
 \begin{equation} \label{impest}
 \int _{B_R}|D^2 v^i_0(x)|^2dx \leq C_n R^{n+2\alpha}.
\end{equation}

\par
Next we describe further properties of the limit functions, $v^1_0$ and  $v^2_0$.
Recall that 
\begin{equation*}
 v^i_k(x)=\frac{u^i(r_k x)-\Pi(u^i(r_k x),1)}{S(r_k)},
\end{equation*}
then we have
\begin{equation*}
 -\Delta v^i_k (x)=\frac{r_k^2}{S(r_k)} \left( -\Delta u^i(r_k x)+tr A^i_{r_k} \right).
\end{equation*}
Let us denote
\begin{align*}
 q^1_k(x)=\frac{p^1_{r_k}(r_k x)-p^2_{r_k}(r_k x)+\varphi^1(r_k x)}{S(r_k)}, \textrm{ and } \\
  q^2_k(x)=\frac{p^2_{r_k}(r_k x)-p^1_{r_k}(r_k x)+\varphi^2(r_k x)}{S(r_k)}.
\end{align*}

Then $ (v^1_k, v^2_k)$ is a strong solution to the following system
\begin{equation*}
\begin{cases}
\min (-\Delta v^1_k-\frac{tr A^1_{r_k}}{k r_k^\alpha}, v^1_k-v^2_k+q^1_k)=0 \\
\min (-\Delta v^2_k-\frac{tr A^2_{r_k}}{k r_k^\alpha}, v^2_k-v^1_k+q^2_k)=0 \\
\min (-\Delta v^1_k-\frac{tr A^1_{r_k}}{k r_k^\alpha},-\Delta v^2_k-\frac{tr A^2_{r_k}}{k r_k^\alpha})=0,
\end{cases}
\end{equation*}
therefore
\begin{equation*}
 -\Delta v^1_k(x)=
 \begin{cases}
  \frac{tr A^1_{r_k}}{k r_k^\alpha}+\frac{\Delta \varphi^1(r_kx)}{k r_k^\alpha},  \textrm{    if }
  r_k x \in \Omega_1 \\
 \frac{tr A^1_{r_k}}{k r_k^\alpha},  \textrm{   otherwise },
 \end{cases}
\end{equation*}
and
\begin{equation*}
 -\Delta v^2_k(x)=
 \begin{cases}
   \frac{tr A^2_{r_k}}{k r_k^\alpha}+\frac{\Delta \varphi^2(r_kx)}{k r_k^\alpha},
   \textrm{    if } r_k x \in \Omega_2 \\
 \frac{tr A^2_{r_k}}{k r_k^\alpha},  \textrm{    otherwise }.
 \end{cases}
\end{equation*}

\par
Then $ \Delta \varphi^i(0) =0$, for $i=1,2$ 
together with $\varphi^i \in C^{2,\alpha}$, implies that 
\begin{equation*}
 \frac{\Arrowvert\Delta \varphi^i(r_k\cdot)\rVert_{L^2(B_R)}}{kr_k^\alpha} \leq 
\frac{C_n}{k} R^{\frac{n}{2}+\alpha} \Arrowvert \varphi^i \rVert_{C^{2,\alpha}} 
\rightarrow 0, \textrm{ as } k\rightarrow\infty,
\end{equation*}
for $i=1,2$, and for any fixed $1\leq R<\infty$.

\par
We have that $v^i_k \rightharpoonup v^i_0 $  weakly in $W^{2,2}(B_R)$ and $ v^i_k \rightarrow v^i_0$ 
in $W^{1,2}(B_R)$, therefore
$ \Delta v^i_k \rightharpoonup \Delta v^i_0$ weakly in $L^{2}(B_R)$, but
$ \Delta v^i_k = \frac{tr A^i_{r_k}}{k r_k^\alpha}
+\frac{\Delta \varphi^i(r_kx)}{k r_k^\alpha} \chi_{\Omega_i}(r_kx)$, and
$\Arrowvert \frac{\Delta \varphi^i(r_kx)}{k r_k^\alpha} \chi_{\Omega_i}(r_kx) \rVert_{L^2(B_R)} \rightarrow 0$.
Thus we may conclude that the sequence of numbers 
$\frac{tr A^i_{r_k}}{k r_k^\alpha}$ converges, and denote 
$ a^i:= \lim_ {k \rightarrow\infty} \frac{tr A^i_{r_k}}{k r_k^\alpha}$.
Then $ \Arrowvert \Delta v^i_k -a^i \rVert_{L^2(B_R)} \rightarrow 0 $ as $k\rightarrow\infty$ 
for every $1\leq R <\infty$. Therefore
both $ -\Delta v^1_0 -a^1 \equiv 0 $ and $ -\Delta v^2_0 -a^2 \equiv 0 $ in $\mathbb{R}^n$.

\par
We have shown that $ v^i_0(x)-a^i \frac{\lvert x \rvert^2}{2n}$ is a harmonic functions in $\mathbb{ R}^n $.
Hence the matrix $ D^2v^i_0$ has harmonic entries $ D^k v^i_0$, where $k$ is a multiindex, $ \lvert k \rvert=2$.
Next we can apply the estimates of the 
derivatives for harmonic functions and inequality \eqref{impest}, to get 
\begin{align*}
 \lvert \nabla D^k v^i_0(x_0) \rvert \leq  R^{ -\frac{n}{2}-1}\lVert D^k v^i_0 \rVert_{L^2(B_R(x_0))} \\
 \leq R^{ -\frac{n}{2}-1} \lVert  D^k v^i_0 \rVert_{L^2(B_{2R})} 
 \leq C' R^{-1+\alpha},
\end{align*}
provided $ R > \lvert x_0\rvert$.
Letting $ R \rightarrow \infty$, we see that 
the derivatives of $D^k v^i_0$ are vanishing, hence $ D^2 v^i_0$ is a constant matrix, and therefore
$v^i_0$, $i \in \{1,2\} $ is a second order polynomial.

\par
According to our construction, 
$v^i_0$  are  orthogonal to the second order polynomials in $L^2(B_1)$-sense, hence
both $ v^1_0$ and $ v^2_0$ must be identically zero. Then the constants $a^1=a^2=0$, and
$ \Arrowvert \Delta v^i_k \rVert_{L^2(B_1)} \rightarrow 0$ as $k\rightarrow\infty$, the latter 
contradicts to the condition $\max_i \Arrowvert D^2v^i_k \rVert_{L^2(B_1)}=1$.
 
\end{proof}

\subsection{$C^{2, \alpha}$- regularity at the meeting points}

We start by showing that the approximating polynomials $p^i_r$ 
converge to a polynomial $p^i_0$, and describe the rate of convergence.

\begin{lemma}
Let $(u^1, u^2)$ be a solution to \eqref{system}, and assume that $\varphi^i \in C^{2, \alpha}$.
Let the polynomials $p^i_r$ be as in the Definition 3, then there exists a polynomial $p^i_0$ such that 
\begin{equation} \label{pol}
 \sup _{x \in B_r}\lvert p^i_r(x)-p^i_0(x)\rvert \leq C r^{2+\alpha}.
\end{equation}
\end{lemma}

\begin{proof}
The condition $\Arrowvert D^2v^i_r \rVert _{L^2(B_1)}\leq 1$ with the inequality 
\eqref{main} implies that 
\begin{equation} \label{1}
 \left( \int_{B_1} \lvert D^2u^i(rx)-D^2p^i_r \rvert ^2 dx \right) ^{\frac{1}{2}} \leq  \frac{S(r)}{r^2}
 \leq C_0 r^{\alpha}
\end{equation}

Recall that $A^i_r=D^2p^i_r$, then using the triangle inequality, and that $A^i_r$ is minimizing 
$ \Arrowvert D^2 u^i(rx) - A \rVert_{L^2(B_1)}$ over matricies $A \in 
\mathbb{R}^n  \times \mathbb{R}^n$, we get 
$ \lvert A^i_r-A^i_{\frac{r}{2}} \rvert \leq C_0 r^{\alpha}$ for all $0<r<1$. 
By taking $r=2^{-n}$, we see that $A^i_{2^{-n}}$ is a Cauchy sequence;
\begin{equation*}
 \lvert A^i_{2^{-n}} - A^i_{2^{-n-m}}\rvert \leq  \Sigma_{k=0}^{m-1} 
 \lvert A^i_{2^{-n-k}} - A^i_{2^{-n-k-1}}\rvert \leq
\end{equation*}
\begin{equation*}
  \Sigma_{k=0}^{m-1}(2^{-\alpha})^{n+k}= 2^{-\alpha n} \Sigma_{k=0}^{m-1}2^{-\alpha k},
\end{equation*}
from the convergence of the series $ \Sigma(2^{-\alpha})^{n}$, it follows that 
$A^i_{2^{-n}}$ converges to some matrix $A^i_0$ as $n$ goes to infinity.
The inequality also provides the rate of convergence; for a fixed $n$, by letting $m$
go to infinity, we see that $ \lvert A^i_{2^{-n}} - A^i_0 \rvert \leq C_0 2^{-n\alpha}$.

\par
Moreover, we get the estimate 
\begin{equation*}
 \lvert A^i_{r} - A^i_0 \rvert \leq C r^{\alpha},
\end{equation*}
for $0<r < 1$, by choosing $n$ so that $2^{-n-1} < r \leq 2^{-n}$.

\par
Next we proceed to describe the rate of convergence of $b^i_r$ and $c^i_r$.
We know that $ b^i_r=(\nabla u ^i )_r$, and $c^i_r=(u^i)_r$, taking into account that $u^i \in C^{1, \gamma}$,
we see that $ b^i_r \rightarrow \nabla u^i(0)$  and $c^i_r\rightarrow u(0)$ as $r \rightarrow 0$. 
Our aim is to show that
actually
\begin{equation*}
\lvert b^i_{r} - \nabla u^i(0) \rvert \leq C r^{1+\alpha}  \textrm{  and }
\lvert c^i_r- u^i(0) \rvert \leq C r^{2+\alpha}.
\end{equation*}

\par
Let us recall that $ (\nabla v^i_r)_1= 0 $, and therefore Poincare's inequality implies that 
\begin{align*}
\Arrowvert \nabla v^i_r \rVert _{L^2(B_1)} \leq C \Arrowvert D^2 v^i_r \rVert _{L^2(B_1)}\leq C'.
\end{align*}
Hence
\begin{equation} \label{2}
 \left( \int_{B_1} \lvert \nabla u^i(rx)-\nabla p^i_r(rx) \rvert ^2 dx \right) ^{\frac{1}{2}} \leq 
   C' \frac{S(r)}{r}
 \leq C r^{1+\alpha}.
\end{equation}

\par
Taking into account that $b^i_r$ is minimizing $ \lVert \nabla u^i(rx)- r x \cdot A^i_r -b \rVert_{L^2(B_1)}$ over 
$ b \in \mathbb{R}^n$, and applying triangle's inequality we get 
$ \lvert b^i_{r} - \nabla u^i(0) \rvert \leq C r^{1+\alpha}$.

\par
Furthermore, using Poincare's inequality once again, we see that
\begin{equation} \label{3}
 \left( \int_{B_1} \lvert u^i(rx)- p^i_r(rx)+r^2 (\Delta u^i)_r \rvert ^2 dx \right) ^{\frac{1}{2}} \leq 
   C' S(r)
 \leq C r^{2+\alpha},
\end{equation}
then 
\begin{equation*}
\lvert c^i_{r} - u^i(0) \rvert \leq C r^{2+\alpha},
\end{equation*}
by using that $c^i_r$
is minimizing $ \Arrowvert u^i(rx)+ r^2(\Delta u^i)_r-\frac{1}{2} r^2 x \cdot A^i_r \cdot x -
r b^i_r \cdot x -c \rVert_{L^2(B_1)}$ over $c \in \mathbb{ R}$.

\par
Finally, after combining our estimates for $ A^i_r, b^i_r \textrm{ and } c^i_r $, we get \eqref{pol},
where 
\begin{equation*}
p^i_0(x)= \frac{1}{2} x \cdot A^i_0 \cdot x + \nabla u^i(0) \cdot x+ u^i(0),
\end{equation*}
for $i= 1,2$.

\end{proof}

\begin{corollary}
Under the assumptions of Lemma 2 it follows that
\begin{equation}
\Arrowvert u^i(rx)-p^i_r(rx) \rVert_{W^{2,2}(B_1)} \leq C_{n,\alpha}
\left( \max_i \Arrowvert \varphi^i \rVert_{C^{2,\alpha}}+ \max_i \Arrowvert u^i \rVert_{L^2}\right)r^{2+\alpha},
\end{equation}
where $ C_{n,\alpha}$ is a dimensional constant.
\end{corollary}

\begin{proof}
Let us recall that $ \Arrowvert \Delta u^i(rx)\rVert_{L^2(B_1)} \leq C \Arrowvert \Delta \varphi^i \rVert_{C^\alpha}r^\alpha$,
then the statement follows from the inequalities \eqref{1}, \eqref{2} and \eqref{3}.
\end{proof}

\par
Now we are ready to prove the main theorem.
\begin{theorem}
 Assume $\varphi^1, \varphi^2 \in C^{2,\alpha }$, and $\mathscr{L}=\overline{\mathscr{L}^0}$
 then the solution to the system \eqref{system}, $(u^1, u^2)$
 is $C^{2, \alpha}$-regular on $ \partial \Omega_1 \cap \partial \Omega_2 \cap \partial \mathscr{L} \cap \Omega$,
 in the sense that for every $x_0 \in \partial \Omega_1 \cap \partial \Omega_2 \cap \partial \mathscr{L}\cap 
 \Omega$,
 there exist second order polynomials $ p^1_{x_0} $, $ p^2_{x_0} $,  such that
 \begin{equation}
 \sup _{x \in B_r(x_0)}\lvert u^i(x)-p^i_{x_0}(x) \rvert \leq C r^{2+\alpha}
\end{equation}
 where the constant $C>0$ depends only on the given data.
\end{theorem}

\begin{proof}
Without loss of generality, we may assume $x_0=0$,
and consider the following rescalings
\begin{equation*}
 v^i_r(x)= \frac{u^i(rx) -p^i_0(rx)}{r^{2+\alpha}}, 
\end{equation*}
then according to Lemma 2 and Corollary 2, $ \Arrowvert v^i_r \rVert_ {W^ {2,2}(B_1)} \leq C$.

\par
The pair $(v^1_r, v^2_r)$ solves the following system

\begin{equation*}
\begin{cases}
\min (-\Delta v^1_r-\frac{tr A^1_0}{r^\alpha}, v^1_r-v^2_r+q^1_r)=0 \\
\min (-\Delta v^2_r- \frac{tr A^2_0}{r^\alpha}, v^2_r-v^1_r+q^2_r)=0 \\
\min (-\Delta v^1_r-\frac{tr A^1_0}{r^\alpha}, -\Delta v^2_r- \frac{tr A^2_0}{r^\alpha} )=0,
\end{cases}
\end{equation*}

where
\begin{equation*}
 q^1_r(x)= \frac{p^1_0(rx)-p^2_0(rx) +\varphi^1(rx)}{r^{2+\alpha}},
 q^2_r(x)= \frac{p^2_0(rx)-p^1_0(rx) +\varphi^2(rx)}{r^{2+\alpha}}.
\end{equation*}

Then
\begin{equation*}
 -\Delta v^i_r=\begin{cases}
    \frac{tr A^i_0}{r^\alpha}+ \frac{\Delta \varphi^i(rx)}{r^\alpha}  \textrm{    if  } rx \in \Omega_i  \\
    \frac{tr A^i_0}{r^\alpha} \textrm{    otherwise.  }
   \end{cases}
\end{equation*}

\par
We assumed that $0 \in \partial \Omega_1 \cap \partial \Omega_2 \cap \partial \mathscr{L}$, then 
$ \Delta \varphi^i (0)=0$, $i=1,2$. Hence 
$ \lvert \frac{\Delta \varphi^1(rx)}{r^\alpha}\rvert \leq
\Arrowvert \varphi^1_r \rVert_{C^{2,\alpha}(B_1)} \lvert x \rvert^\alpha$.
We know that $ \Arrowvert \Delta v^i_r\rVert_{L^2(B_1)}$
is bounded, therefore $tr A^i_0=0$, and $ \Delta v^i_r (x)$ is uniformly bounded.

\par
We have that $ \Arrowvert v^i_r \rVert_{L^2(B_1)} \leq C$
and we saw that $ \Arrowvert \Delta v^i_r \rVert_{L^\infty(B_1)} \leq 
\Arrowvert \varphi^i \rVert_{C^{2,\alpha}(B_1)}$. 
Using the Calderon-Zygmund estimates, we conclude that 
$ \Arrowvert v^i_r \rVert_{C^{1,\gamma}}$
is uniformly bounded. In particular,  $ \lvert v^i_r(x) \rvert \leq 
C'(C_0+\Arrowvert \varphi^i \rVert_{C^{2,\alpha}(B_1)} ) $ for every
$ x \in B_1$ and $r\leq \frac{1}{2} $. 

\par
Recall that we set  $C_0=C \left(\max_ i \Arrowvert \Delta \varphi^i \rVert_{L^\infty}+
\max_i \Arrowvert u^i \rVert_{L^2(B_1)} \right) $, and $ v^i_r(x)= \frac{u^i(rx) -p^i_0(rx)}{r^{2+\alpha}} $.
Then we get the desired inequality
 \begin{equation*}
 \sup _{x \in B_r}\lvert u^i(x)-p^i_0(x) \rvert \leq C
 \left(\max_{i \in \{1,2\}} \Arrowvert u^i \rVert_{L^2(B_1)} +\max_{i\in \{1,2\}}\Arrowvert \varphi^i \rVert_{C^{2,\alpha}(B_1)} \right)
 r^{2+\alpha}.
\end{equation*}
\end{proof}

\subsection{A counterexample in case the zero-loop set has an isolated point}
Here we give a counterexample, showing that if the zero loop set has an isolated point,
then the solution may not be $C^{1,1}$.

\par
We consider the following system in $\mathbb{R}^2$
\begin{equation} \label{countex}
\begin{cases}
\min (-\Delta u^1, u^1-u^2+\varphi)=0 \\
\min (-\Delta u^2, u^2-u^1+\varphi)=0,
\end{cases}
\end{equation}
with $ \varphi= \frac{1}{4} \lvert x \rvert ^2 $.

\par
Then the difference $U=u^1-u^2$ solves the following double-obstacle problem 
in $\mathbb{ R }^2$

\begin{equation} \label{do}
-\Delta U = \begin{cases}
1, \textrm{   if } U=-\varphi \\
-1, \textrm{   if }  U= \varphi \\
0 , \textrm{   if }  -\varphi< U< \varphi,
\end{cases}
\end{equation}
and $-\Delta u^1 =(-\Delta U )^+$ and $ -\Delta u^2 =(\Delta U )^+ $.

\par
Now let us consider a function $ w $ defined as follows
\begin{equation*}
w = \begin{cases}
- \frac{1}{4} \lvert x \rvert^2 , \textrm{   if } x_1> 0, x_2>0 \\
\frac{1}{4} (x_1^2-x_2^2), \textrm{   if }  x_1<0, x_2 >0 \\
\frac{1}{4} (x_2^2-x_1^2), \textrm{   if }  x_1>0, x_2 <0 \\ 
 \frac{1}{4} \lvert x \rvert^2 , \textrm{   if } x_1< 0, x_2<0.
\end{cases}
\end{equation*}
Then $w \in C^{1,1}$ also solves the double-obstacle problem \eqref{do}, therefore we 
choose $ U \equiv w $. 

\par
Next we write $u^1$ explicitly in polar coordinates
\begin{equation*}
u^1(r,\theta) =
\begin{cases}
 -\frac{1}{4}r^2-\frac{1}{2 \pi } r^2 \theta \cos 2\theta -\frac{1}{2 \pi} r^2 \ln r \sin 2\theta,
 \textrm{   if } 0<\theta \leq \frac{\pi }{2}\\
-\frac{1}{4}r^2 \cos 2 \theta +\frac{1}{2 \pi } r^2 \theta \cos 2\theta +\frac{1}{2 \pi} r^2 \ln r \sin 2\theta, \textrm{   otherwise  }, 
\end{cases}
\end{equation*}
here the function   $ r^2 \theta \cos 2\theta + r^2 \ln r \sin 2\theta \in C^{1, \gamma} $   
for every $0<\gamma<1$, solves the Laplace equation in $ \mathbb{ R }^2 \backslash \{0\}$,
but is not $C^{1,1}$ near the origin.

\par
Therefore $u^1$ is a $C^{1,\gamma}$ function in the unit ball in $\mathbb{ R }^2$
but it is not $C^{1,1}$
in the neighborhood of the origin, since $ \lvert \frac{\partial^2 u^1}{\partial r^2}\rvert 
\approx \lvert \ln r \rvert \rightarrow\infty $ as $r\rightarrow 0$.
Moreover, 
  $ -\Delta u^1(r, \theta) = \chi _{\{0<\theta \leq \frac{\pi }{2} \}}=  \chi _{\{ u^1 > 0 \}}$,
provided $r> 0 $ is small enough.

\par
Next we take $ u^2(r, \theta )= u^1(r, \theta)-w$, then 
\begin{equation*}
u^2(r,\theta) =
\begin{cases}
-\frac{1}{2 \pi } r^2 \theta \cos 2\theta -\frac{1}{2 \pi} r^2 \ln r \sin 2\theta,
 \textrm{   if } 0<\theta \leq \frac{\pi }{2}\\
 - \frac{1}{2}r^2 \cos 2\theta+\frac{1}{2 \pi } r^2 \theta \cos 2\theta +\frac{1}{2 \pi} r^2 \ln r \sin 2\theta,
 \textrm{  if } \frac{\pi }{2} <\theta \leq \pi \\
 - \frac{1}{4}r^2- \frac{1}{4}r^2 \cos 2\theta+\frac{1}{2 \pi } r^2 \theta \cos 2\theta +\frac{1}{2 \pi} r^2 \ln r \sin 2\theta,
 \textrm{   if } \pi <\theta \leq \frac{3 \pi }{2}\\
 \frac{1}{2 \pi } r^2 \theta \cos 2\theta +\frac{1}{2 \pi} r^2 \ln r \sin 2\theta,
 \textrm{   if }\frac{3 \pi }{2} <\theta \leq 2 \pi 
\end{cases}
\end{equation*}

Neither $u^1$ nor  $u^2 $ is a $C^{1,1}$ function. However, it is easy to see that $(u^1,u^2) $
solves \eqref{countex}.

\bibliographystyle{plain}
\bibliography{switching}

\end{document}